\documentclass[a4paper,twoside,10pt]{article}
\usepackage[a4paper,left=2.5cm,right=2.5cm, top=2.5cm, bottom=2.5cm]{geometry}
\usepackage[latin1]{inputenc}
\usepackage{cuted}
\usepackage{orcidlink}
\definecolor{darkblue}{rgb}{0.0, 0.0, 0.55}

\usepackage{amssymb}
\usepackage{mathrsfs}
\usepackage{mathtools}
\usepackage{enumitem}
\usepackage{accents}
\usepackage{multirow}
\usepackage{graphicx}
\usepackage{comment}
\usepackage{algorithm2e}
\usepackage{amscd,amsmath,amssymb,mathrsfs,bbm,listings}
\usepackage{cancel}
\allowdisplaybreaks
\graphicspath{{figs/}}
\usepackage{amsmath}
\usepackage{orcidlink}
\usepackage{footmisc}
\usepackage{amsthm}
\usepackage{amssymb}
\usepackage{stmaryrd}
\SetSymbolFont{stmry}{bold}{U}{stmry}{m}{n}\usepackage{bigints}
\usepackage{cite}
\usepackage{color}
\usepackage[abs]{overpic}
\usepackage[font=footnotesize,labelfont=bf]{caption}
\usepackage{cases}
\usepackage{tikz}
\usepackage{rotating}
\usepackage{blkarray}
\usetikzlibrary{matrix,calc,arrows,cd}
\usepackage{soul,xcolor}
\usepackage{enumitem}
\usepackage{verbatim}
\usepackage{graphicx}
\usepackage{subcaption}
\usepackage{mwe}
\usepackage{tikz}
\usepackage{hyperref}
\usepackage{amsthm}

\newtheorem{theorem}{Theorem}[section]
\newtheorem{lemma}[theorem]{Lemma}

\newtheorem{assumption}{Assumption}[section]

\newtheorem{remark}[theorem]{Remark}

\newcommand{\eremk}{\hbox{}\hfill\rule{0.8ex}{0.8ex}}

\newcommand{\R}{\mathbb{R}}
\newcommand{\N}{\mathbb{N}}
\newcommand{\ds}{\, \mathrm{d}s}
\newcommand{\dt}{\, \mathrm{d}t}
\newcommand{\cont}{\mathrm{cont}}

\newcommand{\Tt}{\mathcal{T}_{\tau}}
\newcommand{\In}{I_n}
\newcommand{\Inmo}{I_{n-1}}
\newcommand{\tn}{t_n}
\newcommand{\tnmo}{t_{n-1}}
\newcommand{\Nt}{N_{\tau}}
\newcommand{\tNt}{t_{N_{\tau}}}

\newcommand{\jumpn}[1]{\llbracket #1 \rrbracket_{n}}

\newcommand{\Rt}{\mathcal{R}_{\tau}}
\newcommand{\CR}{C_{\mathcal{R}}}
\newcommand{\Ctr}{C_{\mathrm{tr}}}
\newcommand{\Cinv}{C_{\mathrm{inv}}}
\newcommand{\Csca}{C_{\mathrm{sca}}}
\newcommand{\Cj}{C_J}

\newcommand{\ut}{u_{\tau}}
\newcommand{\vt}{v_{\tau}}

\newcommand{\Vt}{\mathcal{P}_{\tau}}
\newcommand{\calL}{\mathcal{L}}

\newcommand{\calU}{\mathcal{U}}
\newcommand{\calF}{\mathcal{F}}

\newcommand{\Pp}[2]{\mathcal{P}^{#1}(#2)}

\newlength{\dhatheight}

\newcommand{\dpt}{\partial_t}
\newcommand{\Norm}[2]{\|#1\|_{#2}}
\newcommand{\sh}{\sigma_{\delta}}
\newcommand{\jump}[1]{\llbracket #1 \rrbracket}

\newcommand{\bx}{\boldsymbol{x}}

\newcommand{\Ct}{C_{\theta}}

\numberwithin{equation}{section}
\hypersetup{
    colorlinks,
    linktoc=section,
    citecolor=red,
    linkcolor=darkblue,
    urlcolor=magenta,
    pdfborder={0 0 0}
}

\title{A note on the compactness properties of\\ 
discontinuous Galerkin time discretizations} 
\author{Sergio G\'omez\thanks{Department of Mathematics and Applications, University of Milano-Bicocca, Via Cozzi 55, 20125 Milan, Italy (\href{mailto:sergio.gomezmacias@unimib.it}{sergio.gomezmacias@unimib.it})}
\thanks{IMATI-CNR ``E. Magenes", Via Ferrata 5, 27100 Pavia, Italy}\ \orcidlink{0000-0001-9156-5135}}
\date{}

\begin{document}

\maketitle

\begin{abstract}
\noindent This work extends the discrete compactness results of Walkington (SIAM J. Numer. Anal., 47(6):4680--4710, 2010) for high-order discontinuous Galerkin time discretizations of parabolic problems to more general function space settings. 
In particular, we show a discrete version of the Aubin--Lions--Simon lemma that holds for general Banach spaces~$X$, $B$, and~$Y$ satisfying~$X \hookrightarrow B$ compactly and~$B \hookrightarrow Y$ continuously. 
Our proofs rely on the properties of a time reconstruction operator and remove the need for quasi-uniform time partitions assumed in previous works.
Thus, we provide a useful and flexible tool for the analysis of high-order discontinuous Galerkin time discretizations of complex nonlinear partial differential equations.
\end{abstract}

\paragraph{Keywords.} high-order discontinuous Galerkin, discrete compactness, Aubin--Lions--Simon lemma, time reconstruction operator

\paragraph{Mathematics Subject Classification.} 46B50, 35A35, 65M12, 65M60

\section{Introduction}
The aim of this work is to present some discrete compactness results for sequences of (possibly discontinuous) piecewise polynomial functions in time. 
In particular, we show an extension of~\cite[Thm.~3.1]{Walkington:2010} to a more general variational framework. Moreover, our proofs are simpler and avoid the requirement of (global) quasi-uniformity of the time partitions, as they rely on continuous-in-time reconstructions that allow us to use compactness results from the continuous setting. Our main result (Theorem~\ref{thm:main}) can also be seen as a high-order generalization of~\cite[Thm.~1]{Dreher_Jungel:2012}, which addressed the case of piecewise-constant functions in time.

\paragraph{Previous works.} Using discrete compactness for sequences of piecewise-constant functions in time is a standard tool not only to show existence of weak solutions to nonlinear evolution problems, but also to prove convergence of numerical schemes; see, for instance, applications in flow dynamics~\cite{Gallouet_Latche:2012}, nonlinear cross-diffusion systems~\cite{Andreianov_Bendahmane_Ruiz-Baier:2011,Jungel_Zurek:2023,Gomez-Jungel-Perugia:2025}, and hyperbolic--parabolic problems~\cite{Andreianov_Bendahmane_Karlsen:2010}. Such a technique is particularly useful when deriving \emph{a priori} error estimates is too challenging, as well as to show convergence in mesh-independent norms under minimal regularity assumptions.

In contrast, the literature on discrete compactness results for higher-order piecewise polynomial functions in time is scarcer. To the best of our knowledge, Walkington's result in~\cite[Thm.~3.1]{Walkington:2010} was the first of this type. Subsequently, a variation used for the miscible displacement problem was presented in~\cite[Thm.~3.6]{Riviere_Walkington:2011}, and an extension to nonconforming approximations in space was established in~\cite[Thm.~3.2]{Li_Riviere_Walkington:2015}.
The main idea of such results was to show that discontinuous Galerkin (DG) approximations in time are uniformly equicontinuous (see~\cite[Lemma 3.3]{Walkington:2010}), thus circumventing the lack of differentiability of such discrete functions, at the expense of assuming quasi-uniform time partitions.
These results have been used for high-order DG time discretizations of reaction--diffusion problems~\cite{Chrysafinos_Filopoulos_Papathanasiou:2013,Chrysafinos:2019}, incompressible fluid dynamics~\cite{Kirk_Cesmelioglu_Rhebergen:2023,Gazca-Orozco_Kaltenbach:2025}, porous media~\cite{Riviere_Walkington:2011,Li_Riviere_Walkington:2015,Girault_Riviere_Cappanera:2021}, and optimal control~\cite{Chrysafinos:2013,Casas_Chrysafinos:2016}.

Recently, in~\cite[Thm.~A.1]{Kirk_Cesmelioglu_Rhebergen:2023}, it was shown that some of the assumptions in~\cite[Thm.~3.1]{Walkington:2010} and~\cite[Thm.~3.2]{Li_Riviere_Walkington:2015} can be replaced by conditions on the time derivative of the continuous-in-time reconstructions introduced in~\cite[\S2.1]{Makridakis_Nochetto:2006}. Following this approach, we show that the equicontinuity argument can be avoided, as discrete compactness can instead be established by only exploiting the properties of the reconstruction operator.

\paragraph{Compactness results for continuous functions in time.}
We first recall some standard results for the continuous setting, which we use in the proof of Theorem~\ref{thm:main} below. The Aubin--Lions lemma establishes some sufficient conditions to prove that a set of functions~$\calU$ is relatively compact in~$L^p(0, T; B)$, where~$p \in [1, \infty)$, $T > 0$, and~$B$ is a Banach space, 
i.e., under these conditions, any bounded sequence in~$\calU$ has a subsequence that converges strongly in~$L^p(0, T; B)$. The following version from~\cite[Cor.~4]{Simon:1987} is called the Aubin--Lions--Simon lemma, as Simon established the result without the assumption in~\cite[Thm.~1]{Aubin:1963} and\cite[Thm.~I.5.1]{Lions:1969} that the Banach spaces involved are reflexive.

\begin{theorem}[Aubin--Lions--Simon]
\label{thm:aubin-lions-simon}
Let~$X$, $B$, and~$Y$ be Banach spaces with
\begin{equation*}
    X \hookrightarrow B \text{ compactly} \quad \text{and} \quad B \hookrightarrow Y \text{ continuously}.
\end{equation*}
Let also~$p,\, r \in [1, \infty]$, and~$\calU \subset L^p(0, T; X)$ be a set of functions such that
\begin{subequations}
\begin{alignat}{3}
\label{eq:calU}
    \calU & \text{ is bounded in~$L^p(0, T; X)$}, \\
\label{eq:dpt-calU}
    \dpt \calU := \big\{\dpt u\ : \ u \in \calU \big\} & \text{ is bounded in~$L^r(0, T; Y)$}.
\end{alignat}
\end{subequations}
\begin{itemize}
\item If~$1 \le p < \infty$, then~$\calU$ is relatively compact in~$L^p(0, T; B)$.
\item If~$p = \infty$ and~$r > 1$, then~$\calU$ is relatively compact in~$C^0([0, T]; B)$.
\end{itemize}
\end{theorem}

\noindent In~\cite[Thm.~5]{Simon:1987}, Simon showed that condition~\eqref{eq:dpt-calU} on~$\calU$ can be weakened to
\begin{equation}
\label{eq:shifted-condition}
\lim_{\delta \to 0^+} \Norm{\sh u - u}{L^r(0, T - \delta; Y)} = 0, \quad \text{ uniformly for~$u \in \calU$},
\end{equation}
where the shift operator~$\sh u(\cdot, t) := u(\cdot, t + \delta)$ for~$t \in [0, T - \delta]$. 
This version with weaker assumptions is particularly useful when dealing with (possibly discontinuous) piecewise polynomial functions in time, as they are not differentiable and hence condition~\eqref{eq:dpt-calU} cannot hold. However, verifying~\eqref{eq:shifted-condition} for all time shifts~$\delta > 0$ can be cumbersome and conditions that can be deduced directly from the bounds on the discrete solutions are more convenient (see~\cite{Dreher_Jungel:2012}).

Next theorem concerns relative compactness of higher regularity under some slightly more restrictive assumptions; see~\cite[Thm.~1.1]{Amann:2000}.
\begin{theorem}[Amann]
\label{thm:amann}
Let~$X$, $B$, and~$Y$ be Banach spaces such that
\begin{subequations}
\begin{alignat}{3}
\label{eq:amann-1}
X \hookrightarrow B \ \text{ and } \ B \hookrightarrow Y & \text{ continuously}, \\
\label{eq:amann-2}
X \hookrightarrow Y \ & \text{ compactly}, \\
\nonumber
\text{for some~$\theta \in (0, 1)$, there exists~$\Ct > 0$ such that} \\
\label{eq:amann-3}
\Norm{u}{B} \le \Ct \Norm{u}{X}^{\theta} & \Norm{u}{Y}^{1-\theta} \qquad \text{for all~$u \in X$}.
\end{alignat}
\end{subequations}
Let~$\calU \subset L^p(0, T; X)$ be a set of functions satisfying condition~\eqref{eq:calU} and either
\begin{itemize}
    \item condition~\eqref{eq:dpt-calU} (setting~$s = 1$), or
    \item for~$p \in [1, \infty)$, $r \le p$, and some~$0 < s < 1$, there is~$C_s > 0$ such that
    \begin{equation}
        \label{eq:Amann-condition}
        \Norm{\sh u - u}{L^r(0, T - \delta; Y)} \le C_s \delta^{s} \quad \text{ for all~$\delta > 0$ and~$u \in \calU$.}
    \end{equation}
\end{itemize}
Then, $\calU$ is relatively compact in~$L^q(0, T; B)$ for all~$1 \le q < rp/((1-\theta)(1 - s r) p + \theta r)$ provided that $(sr - 1)p/r \le \theta/(1- \theta)$.
\end{theorem}

\begin{remark}
\label{rmk:walkington}
If~$X$ is a Banach space, $B$ is a Hilbert space, and we have a Gelfand triplet~$X \hookrightarrow B \hookrightarrow X'$ with dense compact embeddings, then condition~\eqref{eq:amann-3} holds with~$Y = X'$, $\theta = 1/2$, and~$C_{\theta} = 1$. In this setting, if condition~\eqref{eq:Amann-condition} holds with~$r = 1$ for all~$0 < s < 1$, then Theorem~\ref{thm:amann} implies that~$\calU$ is relatively compact in~$L^q(0, T; B)$ for~$1 \le q < 2p$. 
The result in~\cite[Thm.~3.1(1)]{Walkington:2010} for the high-order DG time stepping can be seen as a discrete counterpart of this particular situation.
\eremk
\end{remark}

\section{Discontinuous Galerkin time discretizations}
In Section~\ref{sec:DG-notation}, we introduce some DG notation and define spaces of piecewise polynomial functions in time. 
In Section~\ref{sec:time-reconstruction}, we present the definition and main properties of the time reconstruction operator used in the proof of 
our discrete compactness result in Section~\ref{sec:compactness}.
A brief discussion of the application of~Theorem~\ref{thm:main} to parabolic problems is 
presented in Section~\ref{sec:parabolic}.

\subsection{DG notation and piecewise polynomial spaces}
\label{sec:DG-notation}
Let~$X$ be a Banach space, and let~$\{\Tt\}_{\tau > 0}$ be a family of partitions of the time interval~$(0, T)$, where each~$\Tt$ is determined by some nodes~$0 =: t_0 < \ldots < \tNt := T$. For~$n = 1, \ldots, N_{\tau}$, we define the time interval~$\In := (\tnmo, \tn)$ and the time step~$\tau_n := \tn - \tnmo$. The subscript~$\tau$ stands for the maximum time step, i.e., $\tau := \max_{n = 1, \ldots, \Nt} \tau_n$. We make the following assumption on~$\{\Tt\}_{\tau > 0}$, which is less restrictive than the (global) quasi-uniformity assumption in~\cite[Thm.~3.1]{Walkington:2010}, \cite[Thm.~3.6]{Riviere_Walkington:2011}, and~\cite[Thm.~3.2]{Li_Riviere_Walkington:2015}. 
In particular, this condition holds for families of geometrically refined partitions in time, which can be used to resolve temporal singularities (see, e.g., \cite{Schotzau_Schwab:2000,Schotzau_Schwab-b:2000}). 

\begin{assumption}[Time steps ratio]
\label{asm:local-quasi-uniformity}
There exists a positive constant~$C_{\star}$ such that, for all~$\tau > 0$ and~$n = 2, \ldots, N_{\tau}$,
\begin{equation*}
\tau_n/\tau_{n - 1} \le C_{\star}.
\end{equation*}
\end{assumption}

For each~$\tau > 0$, let~$X_{\tau}$ be a subspace of~$X$. Given~$\ell \in \N$, we define the following spaces:
\begin{alignat*}{3}
\Vt^{\ell}(X_{\tau}) & := \big\{ v \in L^1(0, T; X_{\tau}) \ : \ v(\cdot, t)_{|_{\In}} \in \Pp{\ell}{\In} \otimes X_{\tau} \text{ for~$n = 1, \ldots, N_{\tau}$}\big\}, \\
\Vt^{\ell, \cont}(X_{\tau}) & := \Vt^{\ell}(X_{\tau}) \cap C^0([0, T]; X_{\tau}),
\end{alignat*}
where~$\Pp{\ell}{\In}$ is the space of polynomials of degree at most~$\ell$ defined on~$\In$, and~$\otimes$ denotes the algebraic product for vector spaces. 

We also define the time jumps~$(\jumpn{\cdot})$ for piecewise smooth scalar functions in time~$(v)$ as follows: 
\begin{equation*}
    \jumpn{v}(\cdot) := v(\cdot, \tn^{+}) - v(\cdot, \tn^{-}) \quad  \text{ for~$n = 1, \ldots, N_{\tau}$, }
\end{equation*}
where
\begin{equation*}
    v(\cdot, \tn^+) := \lim_{\varepsilon \to 0^+} v(\cdot, \tn + \varepsilon) \quad \text{ and } \quad v(\cdot, \tn^{-}) := \lim_{\varepsilon \to 0^+} v(\cdot, \tn - \varepsilon).
\end{equation*}

\subsection{Time reconstruction operator} 
\label{sec:time-reconstruction}
We now recall the definition and some properties of the time reconstruction operator introduced in~\cite[\S2.1]{Makridakis_Nochetto:2006}. 
To do so, we resort to the pointwise-in-time definition in~\cite[Eq.~(33)]{Schotzau_Wihler:2010} (see also~\cite[Lemmas 6 and 7]{Schotzau_Wihler:2010}).

For any Banach space~$Z$ and~$\ell \in \N$, we define the lifting operator~$\calL_{\tau} : Z \to \Vt^{\ell}(Z)$ and the time reconstruction operator~$\Rt : \Vt^{\ell}(Z) \to \Vt^{\ell + 1, \cont}(Z)$ as follows: for~$n = 1, \ldots, N_{\tau}$,
\begin{alignat*}{3}
\calL_{\tau}(z)(\cdot, t) & := \frac{z(\cdot)}{\tau_n} \sum_{i = 0}^{\ell} (-1)^i(2i + 1) L_{n, i}(t) & & \qquad \text{for all~$z \in Z$ and~$t \in \In$},  \\
\Rt \vt(\cdot, t) & := \vt(\cdot, t) - \jump{\vt}_{n - 1}(\cdot) + \int_{\tnmo}^t \calL_{\tau}(\jump{\vt}_{n - 1})(\cdot, s) \ds  & & \qquad \text{for all~$\vt \in \Vt^{\ell}(Z)$ and~$t \in \overline{\In}$}, 
\end{alignat*}
where~$L_{n, i}$ is the~$i$th \emph{mapped Legendre polynomial} defined on the interval~$\In$, and we have set~$\jump{\vt}_0(\cdot) := \vt(\cdot, 0) - v_0$ for some ``initial datum" $v_0 \in Z$. 
The explicit dependence of~$\Rt \vt$ on~$v_0$ will be neglected. The continuity in time of~$\Rt$ was proven in~\cite[Lemma~2.1]{Makridakis_Nochetto:2006}, and the following result from~\cite[Lemma 2.2]{Makridakis_Nochetto:2006} provides an estimate of the error between a piecewise polynomial function in time and its continuous-in-time reconstruction. 

\begin{lemma}[Properties of~$\Rt$]
\label{lemma:Rt}
For any Banach space~$Z$, $\ell \in \N$, and~$p \in [1, \infty]$, there exists a positive constant~$\CR$ depending only on~$p$ and~$\ell$ such that
\begin{equation*}
\Norm{\Rt \vt - \vt}{L^p(0, T; Z)} \le \CR \Big(\sum_{n = 1}^{N_{\tau}} 
\tau_n \Norm{\jump{\vt}_{n-1}}{Z}^p\Big)^{1/p} \qquad \text{for all } \vt \in \Vt^{\ell}(Z),
\end{equation*}
with the usual conventions for the case~$p = \infty$.
\end{lemma}

\subsection{Discrete compactness\label{sec:compactness}}
Henceforth, we assume that the degree of approximation~$\ell \in \N$, the final time~$T > 0$, and the ``initial datum" $u_0 \in X$ are fixed. We are now in a position to prove our discrete compactness result; see Remarks~\ref{rem:initial} and~\ref{rem:stab-Xt} for some extensions. 

\begin{theorem}[Discrete Aubin--Lions--Simon compactness]
\label{thm:main}
Let~$X$, $B$, and~$Y$ be Banach spaces such that the embedding~$X \hookrightarrow B$ is compact, and the embedding~$B \hookrightarrow Y$ is continuous. Let also~$u_0 \in X$, and~$\{\Tt\}_{\tau > 0}$ be a family of partitions of~$(0, T)$ with~$\tau \to 0^+$ satisfying Assumption~\ref{asm:local-quasi-uniformity}. For each~$\tau$, let~$\ut \in \Vt^{\ell}(X_{\tau})$ for some subspace~$X_{\tau}$ of~$X$. 

Assume that
\begin{enumerate}[label = (h\arabic*), ref = (h\arabic*)]
\item \label{h1} $\{\ut\}_{\tau > 0}$ is uniformly bounded in~$L^p(0, T; X)$ for some~$p \in [1, \infty]$.
\item \label{h2} $\{\dpt \Rt \ut\}_{\tau > 0}$ is uniformly bounded in~$L^r(0, T; Y)$ for some~$r \in [1, \infty]$.
\item \label{h3} There exists a positive constant~$\Cj$ independent of~$\tau$ such that~$\sum_{n = 1}^{N_{\tau}} \Norm{\jump{\ut}_{n - 1}}{B}^2 \le \Cj$ for all~$\tau > 0$.
\end{enumerate}
Then, there hold
\begin{enumerate}[label = (\roman*), ref = (\roman*)]
\item \label{i)} If~$1 \le p < \infty$, $\{\ut\}_{\tau > 0}$ is relatively compact in~$L^p(0, T; B)$.
\item \label{ii)} If~$p = \infty$ and~$r  > 1$, there exists a subsequence of~$\{\ut\}_{\tau > 0}$ that converges strongly in~$L^q(0, T; B)$ for all~$1 \le q < \infty$ to a function belonging to~$C^0([0, T]; B)$.
\item \label{iii)} If conditions~\eqref{eq:amann-1}--\eqref{eq:amann-3} hold, $\{\ut\}_{\tau > 0 }$ is relatively compact in~$L^q(0, T; B)$ for all~$1 \le q < rp/((1 - \theta) (1 - r) p + \theta r)$ provided that~$(r - 1)p/r \le \theta/(1 - \theta)$.
\end{enumerate}
\end{theorem}
\begin{proof}
We split the proof into three parts. 

\paragraph{Part I) Uniform boundedness of~$\{\Rt \ut\}_{\tau > 0}$.} 
We first show that the time reconstruction~$\Rt \ut$ inherits the uniform boundedness of~$\ut$ in~$L^p(0, T; X)$. Below, we focus on the case~$1 \le p < \infty$, but the case~$p = \infty$ follows analogously. 

Using the triangle inequality and Lemma~\ref{lemma:Rt}, we obtain
\begin{alignat}{3}
\nonumber
\Norm{\Rt \ut}{L^p(0, T; X)} & \le \Norm{\ut}{L^p(0, T; X)} + \Norm{\Rt \ut - \ut}{L^p(0, T; X)} \\
\label{eq:triangle-Rt}
& \le \Norm{\ut}{L^p(0, T; X)} + \CR \Big(\sum_{n = 1}^{N_{\tau}} \tau_n \Norm{\jump{\ut}_{n-1}}{X}^p\Big)^{1/p}.
\end{alignat}
By a standard inverse-trace inequality (see, e.g., \cite[Lemma 12.8]{Ern_Guermond-I:2020}) extended to Bochner 
spaces, there exists a positive constant~$\Ctr$ independent of~$\Tt$ such that
\begin{alignat*}{3}
\Norm{\jump{\ut}_{n - 1}}{X} & \le \Ctr \big(\tau_n^{-1/p} \Norm{\ut}{L^p(\In; X)} + \tau_{n-1}^{-1/p} \Norm{\ut}{L^p(\Inmo; X)}\big) \quad \text{ for~$n = 2, \ldots, N_{\tau}$}, \\
\Norm{\jump{\ut}_0}{X} & \le \Ctr \tau_1^{-1/p}\Norm{\ut}{L^p(I_1; X)} +  \Norm{u_0}{X},
\end{alignat*}
which, together with the triangle inequality and Assumption~\ref{asm:local-quasi-uniformity}, implies
\begin{alignat}{3}
    \nonumber
    \Norm{\Rt \ut - \ut}{L^p(0, T; X)} & \le \CR \bigg[\sum_{n = 2}^{N_{\tau}} \big(\Ctr \Norm{\ut}{L^p(\In; X)} + \Ctr (\tau_n/\tau_{n - 1}) \Norm{\ut}{L^p(\Inmo; X)} \big)^p \\
    \nonumber
    & \quad + \big(\Ctr \Norm{\ut}{L^p(I_1; X)} + \tau_1 \Norm{u_0}{X} \big)^p\bigg]^{1/p}\\
    \label{eq:diff-Rt-ut}
    & \le \CR \big(\Ctr(1 + C_{\star}) \Norm{\ut}{L^p(0, T; X)} + \tau_1^{1/p} \Norm{u_0}{X}\big).
\end{alignat}
The uniform boundedness of~$\{\Rt \ut\}_{\tau > 0}$ in~$L^p(0, T; X)$ then follows from hypothesis~\ref{h1} by combining inequalities~\eqref{eq:triangle-Rt} and~\eqref{eq:diff-Rt-ut}.

\paragraph{Part II) Relative compactness of~$\{\Rt \ut\}_{\tau > 0}$.}
Since~$\{\dpt \Rt \ut\}_{\tau > 0}$ is uniformly bounded in~$L^r(0, T; Y)$ by hypothesis~\ref{h2}, and~$\{\Rt \ut\}_{\tau > 0}$ is uniformly bounded in~$L^p(0, T; X)$ by Part I) of this proof, we can use Theorems~\ref{thm:aubin-lions-simon} and~\ref{thm:amann} to deduce that the relative compactness results~\ref{i)} and~\ref{iii)} hold replacing~$\{\ut\}_{\tau > 0}$ by~$\{\Rt \ut\}_{\tau > 0}$. Moreover, if~$p = \infty$ and~$r > 1$, the sequence~$\{\Rt \ut\}_{\tau > 0}$ is relatively compact in~$C^0([0, T]; B)$ by Theorem~\ref{thm:aubin-lions-simon}.

\paragraph{Part III) Relative compactness of~$\{\ut\}_{\tau > 0}$.}
We focus on cases~\ref{i)} and~\ref{ii)}, as case~\ref{iii)} can be proven analogously. 

\paragraph{Proof of case~\ref{i)}.} Assume that~$p \in [1, \infty)$ and~$r \in [1, \infty]$.
From Part II) of this proof, there exists a subsequence, still denoted by~$\{\Rt \ut\}_{\tau > 0}$, that converges strongly in~$L^p(0, T; B)$ to some function~$u^* \in L^p(0, T; B)$. Using the triangle inequality and Lemma~\ref{lemma:Rt}, we have
\begin{alignat}{3}
\nonumber
\Norm{\ut - u^*}{L^p(0, T; B)} & \le \Norm{\ut - \Rt \ut}{L^p(0, T; B)} + \Norm{\Rt \ut - u^*}{L^p(0, T; B)} \\
\label{eq:triangle-ut-u*}
& \le \CR \Big(\sum_{n = 1}^{N_{\tau}}\tau_n \Norm{\jump{\ut}_{n-1}}{B}^p \Big)^{1/p} + \Norm{\Rt \ut - u^*}{L^p(0, T; B)}.
\end{alignat}
Therefore, it only remains to show that the first term on the right-hand side of~\eqref{eq:triangle-ut-u*} converges to~$0$ as~$\tau \to 0^+$. 

For~$p \geq 2$, the following bound can be obtained from hypothesis~\ref{h3} and the standard inequality for vector norms~$\Norm{\bx}{p} \le \Norm{\bx}{2}$:
\begin{equation*}
\CR \Big(\sum_{n = 1}^{N_{\tau}}\tau_n \Norm{\jump{\ut}_{n-1}}{B}^p \Big)^{1/p} \le \CR \sqrt{\Cj} \tau^{1/p} \to 0 \quad \text{as~$\tau \to 0^+$}. 
\end{equation*}

As for the case~$1 \le p < 2$, we use the H\"older inequality (with~$\gamma = 2/(2-p)$ and~$\gamma' = 2/p$) to get
\begin{alignat*}{3}
\CR \Big(\sum_{n = 1}^{N_{\tau}}\tau_n \Norm{\jump{\ut}_{n-1}}{B}^p \Big)^{1/p} & \le \CR \bigg[\Big(\sum_{n = 1}^{N_{\tau}} \tau_n^{\gamma/\gamma} \Big)^{1/\gamma}\Big(\sum_{n = 1}^{N_{\tau}} \tau_n^{(1 - \frac{1}{\gamma})\gamma'}\Norm{\jump{\ut}_{n - 1}}{B}^{\gamma' p} \Big)^{1/{\gamma'}} \bigg]^{1/p} \\
& \le \CR T^{\frac{2-p}{2p}}\tau^{1/2} \Big(\sum_{n = 1}^{N_{\tau}} \Norm{\jump{\ut}_{n - 1}}{B}^2 \Big)^{1/2}  \le \CR T^{\frac{2-p}{2p}} \sqrt{\Cj} \tau^{1/2} \to 0 \quad \text{as~$\tau \to 0^+$}.
\end{alignat*}
This completes the proof of case~\ref{i)}.

\paragraph{Proof of case~\ref{ii)}.} For this case, we proceed similarly as in~\cite[Thm.~1]{Dreher_Jungel:2012}. 

Assume that~$p = \infty$ and~$r > 1$. From Part II) of this proof, there exists a subsequence, still denoted by~$\{\Rt \ut\}_{\tau > 0}$, that converges strongly in~$C^0([0, T]; B)$ to some function~$u^{\star} \in C^0([0, T];B)$. Moreover, using case~\ref{i)} with~$p = q \in [1, \infty)$ instead of~$p = \infty$, there exists another subsequence, still denoted by~$\{\ut\}_{\tau > 0}$, which converges strongly in~$L^q(0, T; B)$ to some function~$\widehat{u} \in L^q(0, T; B)$. Using the triangle inequality, for all~$\tau > 0$ in the intersection of the two subsequences, we have
\begin{alignat}{3}
\nonumber
    \Norm{u^{\star} - \widehat{u}}{L^q(0, T; B)} & \le \Norm{u^{\star} - \Rt \ut}{L^q(0, T; B)} + \Norm{\Rt \ut - \ut}{L^q(0, T; B)} + \Norm{\ut - \widehat{u}}{L^q(0, T; B)} \\
    \label{eq:ustar-hatu}
    & \le T^{1/q} \Norm{u^{\star} - \Rt \ut}{C^0([0, T]; B)} + \Norm{\Rt \ut - \ut}{L^q(0, T; B)} + \Norm{\ut - \widehat{u}}{L^q(0, T; B)}.
\end{alignat}
All the terms on the right-hand side of~\eqref{eq:ustar-hatu} 
converge to~$0$ as~$\tau \to 0^+$, which implies that~$u^{\star}(\cdot, t) = \widehat{u}(\cdot, t)$ in~$B$ for a.e. $t \in (0, T)$. 
\end{proof}

\subsection{Insights into parabolic problems\label{sec:parabolic}}
We now give a brief account of how the assumptions of Theorem~\ref{thm:main} can be verified for the DG time discretization of parabolic problems. 
Assume that~$X \hookrightarrow B \hookrightarrow X'$ is a Gelfand triplet, where 
$X$ is a separable reflexive Banach space with dense compact embedding~$X \hookrightarrow B$, and~$B$ is a separable Hilbert space with inner product~$(\cdot, \cdot)_B$ and norm~$\Norm{\cdot}{B}$. 

Given~$p \in (1, \infty)$ with~$p' := p/(p - 1)$, a (non)linear operator~$A : X \to X'$, and a nonlinear source term~$f : X \to B$, we consider the following first-order evolution equation: find~$u \in W^p(0, T; X) := \{v \in L^p(0, T; X)\, : \, \dpt v \in L^{p'}(0, T; X')$\} such that
\begin{subequations}
\label{eq:abstract-parabolic}
\begin{alignat}{3}
\int_0^T \langle \dpt u, v\rangle_{X' \times X} \dt  + \int_0^T \langle A(u), v \rangle_{X' \times X} \dt & = \int_0^T \big(f(u) , v\big)_B \dt  & & \qquad \text{for all~$v \in L^p(0, T; X)$}, \\
\label{eq:abstract-parabolic-2}
u(\cdot, 0) & = u_0 & & \qquad \text{in~$B$},
\end{alignat}
\end{subequations}
where~$\langle\cdot, \cdot \rangle_{X' \times X}$ denotes the duality product between~$X'$ and~$X$, and the continuous embedding~$W^p(0, T; X) \hookrightarrow C^0([0, T]; B)$ (see, e.g., \cite[Lemma 7.1]{Roubicek:2013}) ensures that~\eqref{eq:abstract-parabolic-2} makes sense.

For a given discrete subspace~$X_{\tau}$ of~$X$, and a degree of approximation~$\ell \in \N$, the DG time discretization of the abstract evolution problem~\eqref{eq:abstract-parabolic} reads: find~$\ut \in \Vt^{\ell}(X_{\tau})$ such that, for all~$\vt \in \Vt^{\ell}(X_{\tau})$,
\begin{equation}
\label{eq:DG-time}
\begin{split}
    \sum_{n = 1}^{N_{\tau}} \Big( \int_{\In} (\dpt \ut, \vt)_B \dt & + \big(\jump{\ut}_{n-1}, \vt(\cdot, \tnmo^+) \big)_B \Big) + \int_0^T a(\ut, \vt) \dt = \int_0^T (f(\ut),  \vt)_{B} \dt,
\end{split}
\end{equation}
where~$a : X \times X \to \R$ is defined by~$a(u, v) := \langle A(u), v\rangle_{X' \times X}$ for all~$u, v \in X$. 

If there is a positive constant~$C_a$ such that~$a(u, u) \geq C_a \Norm{u}{X}^p$ for all~$u \in L^p(0, T; X)$, and~$f(\cdot)$ satisfies an appropriate continuity or growth condition, then one can obtain a stability bound of the form
\begin{equation*}
    \frac12 \Big(\Norm{\ut(\cdot, T)}{B}^2 + \sum_{n = 1}^{N_{\tau}} \Norm{\jump{\ut}_{n-1}}{B}^2 \Big) + C_a \Norm{\ut}{L^p(0, T; X)}^p \le C(T, u_0, f),
\end{equation*}
thus verifying hypotheses~\ref{h1} and~\ref{h3}. As for hypothesis~\ref{h2}, we can use the equivalent definition in~\cite[Eq.~(13)]{Makridakis_Nochetto:2006} of~$\Rt \ut$ to rewrite~\eqref{eq:DG-time} as follows: find~$\ut \in \Vt^{\ell}(X_{\tau})$ such that, for all~$\vt \in \Vt^{\ell}(X_{\tau})$,
\begin{equation}
\label{eq:short}
\begin{split}
    \int_0^T (\dpt \Rt \ut, \vt)_B \dt & = \sum_{n = 1}^{N_{\tau}} \Big( \int_{\In} (\dpt \ut, \vt)_B \dt  + \big(\jump{\ut}_{n-1}, \vt(\cdot, \tnmo^+) \big)_B \Big)  = \int_0^T \langle \calF_{\tau}, \vt\rangle_{X_{\tau}' \times X_{\tau}} \dt,
\end{split}
\end{equation}
where~$\langle \cdot, \cdot \rangle_{X_{\tau}' \times X_{\tau}}$ denotes the duality product between~$X_{\tau}'$ and~$X_{\tau}$, and the functionals~$\{\calF_{\tau}\}_{\tau>0}$, stemming from the terms involving~$a(\cdot, \cdot)$ and~$f(\cdot)$,  are assumed to satisfy that the norms $\{\Norm{\calF_{\tau}}{L^{p'}(0,T;X_\tau')}\}_{\tau>0}$ are uniformly bounded.

We denote by $\Pi_{X_{\tau}} : B \to X_{\tau}$ the~$B$-orthogonal projection onto~$X_{\tau}$, and assume that its restriction to~$X$ is stable in the~$X$ norm (i.e., $\Norm{\Pi_{X_{\tau}} \phi}{X} \lesssim \Norm{\phi}{X}$ for all~$\phi \in X$). We also denote by~$\pi_t : L^1(0, T) \to \Vt^{\ell}(\Tt)$ the~$L^2(0, T)$-orthogonal projection onto~$\Vt^{\ell}(\Tt)$, and define the mixed projection~$\Pi_{X_{\tau}}^t := \pi_t \circ \Pi_{X_{\tau}}$. 
Then, since~$\dpt \Rt \ut \in \Vt^{\ell}(X_{\tau})$, 
it follows from equation~\eqref{eq:short} and the orthogonality properties of~$\pi_t$ and~$\Pi_{X_{\tau}}$ that, for any~$\phi \in L^{p}(0, T; X)$, 
\begin{equation}
\label{eq:stab-dpt-Rt-ut}
\begin{split}
    \int_0^T (\dpt \Rt \ut, \phi)_{B} \dt = \int_0^T (\dpt \Rt \ut, \Pi_{X_{\tau}}^t \phi)_B \dt 
    & = \int_0^T \langle \calF_{\tau}, \Pi_{X_{\tau}}^t \phi \rangle_{X_{\tau}' \times X_{\tau}} \dt  \\
    & \le \Norm{\calF_{\tau}}{L^{p'}(0, T; X_{\tau}')} \Norm{\Pi_{X_{\tau}}^t\phi}{L^{p}(0, T; X)},
\end{split}
\end{equation}
which, combined with the stability properties of~$\pi_t$ (see, e.g., \cite[Thm.~18.16(ii)]{Ern_Guermond-I:2020}) and the stability of~$\Pi_{X_{\tau}}$ in the~$X$ norm, implies that the sequence~$\{\dpt \Rt \ut\}_{\tau > 0}$ is uniformly bounded in~$L^{p'}(0, T; X')$, thus verifying~\ref{h2} with~$Y = X'$ and~$r = p'$. Therefore, provided that~$u_0 \in X$, the sequence~$\{\ut\}_{\tau > 0}$ is relatively compact in~$L^q(0, T; B)$ for all~$1 \le q < \infty$ by Theorem~\ref{thm:main} case~\ref{iii)} with~$r = p'$ and~$\theta = 1/2$; cf. \cite[Thm.~3.1(2)]{Walkington:2010}.

We conclude this section with remarks on some extensions of Theorem~\ref{thm:main}.
\begin{remark}[Initial datum~$u_0 \in B$] 
\label{rem:initial}
In Part I) of the proof of Theorem~\ref{thm:main}, we have used~$u_0 \in X$ to show that the time reconstructions~$\{\Rt \ut\}_{\tau > 0}$ are uniformly bounded in~$L^p(0, T; X)$. On the other hand, the continuous problem just requires~$u_0 \in B$. 
This seems to be a limitation of the argument based on time reconstructions when general subspaces~$X_{\tau} \subset X$ are considered. 
Since our aim is the numerical approximation of the solution to evolution problems, we focus on the practically relevant case of discrete finite element subspaces~$X_{\tau}$, which are characterized by a spatial meshsize parameter~$h > 0$. 

In this finite element setting, the results in Theorem~\ref{thm:main} for~$1 \le p < \infty$ can still be obtained as follows.
Let~$h_{\min}$ be the minimum element diameter, and assume an inverse estimate of the form
$$\Norm{\phi_h}{X} \le \Cinv h_{\min}^{-s} \Norm{\phi_h}{B} \quad \forall \phi_h \in X_{\tau},$$ 
for some~$s > 0$ and a positive constant~$\Cinv$ independent of~$h$ and~$\tau$. Then, take~$\Pi_{X_{\tau}} u_0$ as the initial datum of~$\Rt \ut$ and impose the scaling
$$\tau_1 \le \Csca h_{\min}^{sp},$$ 
for some positive constant~$\Csca$ independent of~$h$ and~$\tau$. 
The corresponding
term~$\tau_1^{1/p} \Norm{\Pi_{X_{\tau}} u_0}{X}$ in Part~I) of the proof of Theorem~\ref{thm:main} can then be bounded  by
\begin{equation*}
\tau_1^{1/p} \Norm{\Pi_{X_{\tau}} u_0}{X} \le C_{\mathrm{inv}} \tau_1^{1/p} h_{\min}^{-s} \Norm{\Pi_{X_{\tau}} u_0}{B} \le \Cinv \Csca \Norm{u_0}{B},
\end{equation*}
which yields the uniform boundedness of~$\{\Rt \ut\}_{\tau > 0}$ in~$L^p(0, T; X)$ when~$u_0 \in B$. The mild condition~$\tau_1 \le \Csca h_{\min}^{sp}$ is required on the first time step only  if~$u_0 \not\in X$, and just determines the scaling of~$\tau_1$ with respect to~$h_{\min}$. Moreover, this is not a CFL condition, as the constant~$\Csca$ does not have to be ``sufficiently small". 

For the classical setting with~$p = 2$, $X = H_0^1(\Omega)$, $B = L^2(\Omega)$, and~$Y = H^{-1}(\Omega)$ (where~$\Omega$ is a given spatial domain), the well-known inverse estimate (see, e.g., \cite[Lemma~12.1 and Eq.~(22.39)]{Ern_Guermond-I:2020})
$$\Norm{\phi_h}{H_0^1(\Omega)} \le \Cinv h_{\min}^{-1} \Norm{\phi_h}{L^2(\Omega)} \quad \forall \phi_h \in X_{\tau},$$ 
leads to the parabolic-type restriction~$\tau_1 \le \Csca h_{\min}^2$. 
This is reasonable, as one has to accurately approximate the  singular behavior of the initial-layer that results when~$u_0$ only belongs to~$L^2(\Omega)$. Such a parabolic scaling is also motivated by interpolation estimates for time-singular parabolic solutions; see \cite{Stevenson_Storn:2023}. Consequently, we consider this restriction to be mild compared to the global quasi-uniformity of the time partitions required in previous works.
\eremk
\end{remark}

\begin{remark}[Stability of~$\Pi_{X_{\tau}}$ and further extensions]
\label{rem:stab-Xt}
The requirement of the stability of the $B$-orthogonal projection~$\Pi_{X_{\tau}}$ in the~$X$ norm arises from the fact that condition~\eqref{eq:dpt-calU} requires a uniform bound of~$\dpt \calU$ in~$L^r(0, T; Y)$ for a fixed Banach space~$Y$. To avoid such a constraint, one can rely on the following simplified version of~\cite[Prop.~9]{Chen_Jungel_Liu:2014} (cf.~\cite[Thm.~C.8 in App. C]{Droniou_etal:2018}): let~$\{X_m\}_{m \in \N}$ and~$\{Y_m\}_{m \in \N}$ be sequences of Banach spaces with~$X_m \subset Y_m$ such that
\begin{enumerate}[label = \alph*), ref = \alph*)]
    \item \label{a)} for any set of functions~$\{\phi_m\}_{m \in \N}$ with~$\phi_m \in X_m$, a uniform bound of~$\Norm{\phi_m}{X_m}$ implies that~$\{\phi_m\}_{m \in \N}$ is relatively compact in~$B$;
    \item \label{b)} $B \hookrightarrow Y_m$ continuously for all~$m \in \N$, with uniform continuity constant;
    \item \label{c)} for some~$ p \in [1, \infty]$, $u_m \in L^p(0, T; X_m)$ for all~$m \in \N$, and~$\Norm{u_m}{L^p(0, T; X_m)}$ is bounded uniformly in~$m \in \N$;
    \item \label{d)}$\Norm{\sigma_{\delta} u_m - u_m}{L^p(0, T - \delta; Y_m)} \to 0$ as~$\delta \to 0$, uniformly in~$m \in \N$.
\end{enumerate}
Then, $\{u_m\}_{m \in \N}$ is relatively compact in~$L^p(0, T; B)$ (if~$1 \le p < \infty$), or in~$C^0([0, T]; B)$ (if~$p = \infty$). 

This result can be easily used to improve Theorem~\ref{thm:main} in two directions:
\begin{itemize}
    \item \textbf{Nonconforming spatial discretizations:} conditions~\ref{a)} and~\ref{c)} no longer require that~$X_{\tau} \subset X$ for some fixed Banach space~$X$. In particular, condition~\ref{a)} can be verified for nonconforming finite element spaces by using known compact embeddings (see, e.g., \cite[\S5]{Buffa_Ortner:2009} and~\cite[\S6.2]{DiPietro_Ern:2010}).
 
    \item \textbf{Relaxing hypothesis~\ref{h2}:} assuming instead that~$\{\Norm{\dpt \Rt \ut}{L^r(0, T; Y_{\tau})}\}_{\tau > 0}$ is uniformly bounded for some~$r \in [1, \infty]$, and~$\{Y_{\tau}\}_{\tau > 0}$ are some Banach spaces with~$B \hookrightarrow Y_{\tau}$ continuously with uniform continuity constant. 
    This is possible as~\cite[Lemma 4]{Simon:1987} states that, for all~$\delta > 0$,
    \begin{equation*}
        \Norm{\sigma_{\delta} \Rt \ut - \Rt \ut }{L^p(0, T-\delta; Y_{\tau})} \le \begin{cases}
            \delta^{1 + 1/p - 1/r} \Norm{\dpt \Rt \ut}{L^r(0, T; Y_\tau)} & \text{if~$r \le p \le \infty$}, \\
            \delta T^{1/p - 1/r} \Norm{\dpt \Rt \ut}{L^r(0, T; Y_{\tau})} & \text{if~$1 \le p \le r$}.
        \end{cases}
    \end{equation*}
    Therefore, the sequence~$\{\Rt \ut\}_{\tau > 0}$ satisfies condition~\ref{d)} for any~$1 \le p < \infty$ (if~$r \in [1, \infty]$), and for~$p = \infty$ (if~$r > 1$).
    Consequently, one can avoid the stability assumption of the~$B$-orthogonal projection~$\Pi_{X_{\tau}}$ in the~$X$ norm, which we used in bound~\eqref{eq:stab-dpt-Rt-ut}.
\end{itemize}
This further extends the applicability of Theorem~\ref{thm:main}.
\eremk
\end{remark}

\section{Concluding remarks}
We have presented a compactness result for high-order piecewise polynomial functions in time, which serves as a discrete counterpart of the Aubin--Lions--Simon lemma and of Amann's theorem on compactness of higher regularity.
By exploiting the properties of a time reconstruction operator, we have circumvented the need for equicontinuity arguments and the assumption of quasi-uniform time partitions  required in earlier works. The case of nonsmooth initial data can be handled in the finite element setting under a suitable scaling condition between the first time step~$\tau_1$ and the spatial meshsize~$h$.
Our approach is expected to extend naturally to nonconforming spatial discretizations and to contribute to the analysis of discontinuous Galerkin time discretizations of complex nonlinear evolution problems.

\paragraph{Acknowledgements.}
The author thanks Keegan Kirk (George Mason University) for pointing out reference~\cite[App.~C]{Droniou_etal:2018}.
This research was partially supported by the European Union (ERC Synergy, NEMESIS, project number
101115663). Views and opinions expressed are, however, those of the author only and do not necessarily
reflect those of the EU or the ERC Executive Agency.
The author is also member of the INdAM-GNCS group.

\end{document}